\documentclass[12pt,a4paper]{article}
\usepackage[english]{babel}
\usepackage{amssymb}
\usepackage{amsmath}
\usepackage{amsthm}
\usepackage{graphicx}
\usepackage[retainorgcmds]{IEEEtrantools}

\theoremstyle{definition}
\newtheorem{definition}{Definition}[section]
\newtheorem{example}[definition]{Example}

\newtheorem{remark}[definition]{Remark}
\theoremstyle{plain}
\newtheorem{lemma}[definition]{Lemma}
\newtheorem{proposition}[definition]{Proposition}
\newtheorem{theorem}[definition]{Theorem}

\DeclareMathOperator{\TP}{\mathbb{T}\mathbb{P}}
\DeclareMathOperator{\val}{val}

\DeclareMathOperator{\conv}{conv}
\DeclareMathOperator{\tropdet}{tropdet}

\author{Filip Cools\footnote{K.U.Leuven, Department of
Mathematics, Celestijnenlaan 200B, B-3001 Leuven, Belgium; email address:
Filip.Cools@wis.kuleuven.be; the author is a postdoctoral fellow of the Research Foundation - Flanders (FWO).}}
\title{Linear pencils of tropical plane curves}
\date{}

\begin{document}
\maketitle {\footnotesize \emph{\textbf{Abstract.---} Analogously as in classical algebraic geometry, linear pencils of tropical plane curves are parameterized by tropical lines in a coefficient space. A special example of such a linear pencil is the set of tropical plane curves with an $n$-element support set through a general configuration of $n$ points in the tropical plane. In \cite{RGST}, it is proved that these linear pencils are compatible with their support set. In this article, we give a characterization of points lying in the fixed locus of a tropical linear pencil and show that each compatible linear pencil comes from a general configuration. \\ \\
\indent \textbf{MSC 2010.---} 14T05, 52B20, 14H50, 14C20}}
\\ ${}$

\section{Introduction}

Let $K=\mathbb{C}\{\{t\}\}$ be the field of Puisseux series (or any other algebraically closed field with a non-trivial valuation), and let $\mathcal{A}=\{a_1,\ldots,a_n\}$ be a subset of $\{(r,s,t)\in(\mathbb{Z}_{\geq 0})^3\,:\,r+s+t=d\}$ for some integer $d$, with $a_i=(r_i,s_i,t_i)$. We always assume that the convex hull $\conv(\mathcal{A})$ is two-dimensional, hence $n\geq 3$. A polynomial $f$ of the form $$\sum_{i\in\{1,\ldots,n\}}\,k_i X^{r_i}Y^{s_i}Z^{t_i}\in K[X,Y,Z]$$ defines a projective plane curve $V(f)$ of degree $d$ in $\mathbb{P}^2(K)$. The \textit{tropical polynomial} corresponding to $f$ is the piece-wise linear map $$F_c:\mathbb{R}^3\to\mathbb{R}:(x,y,z)\mapsto \min_{i\in\{1,\ldots,n\}}\,\{c_i+r_ix+s_iy+t_iz\}$$ with $c=(c_1,\ldots,c_n)=(\val(k_1),\ldots,\val(k_n))\in\mathbb{R}^n$, which is obtained by replacing all the operations in $f$ by the tropical operations $\oplus$ and $\otimes$ (where $x\oplus y:=\min\{x,y\}$ and $x\otimes y:=x+y$ for all $x,y\in\mathbb{R}$) and all the coefficients by its valuations. Let $\mathcal{T}(F_c)$ be the set of all points $(x,y,z)\in\mathbb{R}^3$ for which the minimum $F_c(x,y,z)$ is attained at least twice. A theorem due to Kapranov tells us that $\mathcal{T}(F_c)$ is equal to closure in $\mathbb{R}^3$ of $$\{(\val(X),\val(Y),\val(Z))\,|\,(X,Y,Z)\in V(f)\cap (K^{\times})^3\}.$$ Since $\mathcal{T}(F_c)$ is closed under tropical scalar multiplication, i.e. $$\lambda\otimes(x,y,z)=(\lambda+x,\lambda+y,\lambda+z)\in \mathcal{T}(F_c) \quad\Longleftrightarrow\quad (x,y,z)\in\mathcal{T}(F_c),$$ we can identify $\mathcal{T}(F_c)$ with its image in the \textit{tropical projective plane} $\TP^2=\mathbb{R}^3/\mathbb{R}(1,1,1)$ and we say that $\mathcal{T}(F_c)$ is a \textit{tropical projective plane curve} with support set $\mathcal{A}$. Consider the convex hull of the points $(r_i,s_i,t_i,c_i)\in\mathbb{R}^4$ with $i\in\{1,\ldots,n\}$. Under the projection to the first three coordinates, the lower faces of this polytope map onto the convex hull $\conv(\mathcal{A})$ of $\mathcal{A}$ and give rise to the \textit{regular subdivision} $\Delta(c)$ of $\conv(\mathcal{A})$. Note that two elements $a_i,a_j\in \mathcal{A}$ are connected by a line segment in $\Delta(c)$ if and only if there exists a point $(x,y,z)\in \TP^2$ such that $$c_i+r_ix+s_iy+t_iz=c_j+r_jx+s_jy+t_jz<c_k+r_kx+s_ky+t_kz$$ for all $k\in \{1,\ldots,n\}\setminus\{i,j\}$, and that the triangle $T=\conv\{a_i,a_j,a_k\}$ is a face of $\Delta(c)$ if and only if there exists a point $(x,y,z)\in\TP^2$ such that $$c_i+r_ix+s_iy+t_iz=c_j+r_jx+s_jy+t_jz=c_k+r_kx+s_ky+t_kz<c_{\ell}+r_{\ell}x+s_{\ell}y+t_{\ell}z$$ for all $\ell\in \{1,\ldots,n\}\setminus\{i,j,k\}$. So we can see that the curve $\mathcal{T}(F_c)\subset \TP^2$ is an embedded graph which is dual to $\Delta(c)$.

Tropical plane curves going through some fixed points have been studied intensively, mainly because of its applications in enumerative algebraic geometry (see for example \cite{Mikh,GaMa}). The space of tropical plane quadrics through two or three points is examined in \cite{BrSt}. In this article, we focuss on tropical plane curves with support set $\mathcal{A}$ passing through $n-2$ points in general position. In classical algebraic geometry (and in particularly over the field $K$), the set of such curves is a linear pencil parameterized by a line in the coefficient space. The notion of a linear pencil also makes sense in the tropical setting. Indeed, let $L\subset \TP^{n-1}=\mathbb{R}^n/\mathbb{R}(1,\ldots,1)$ be a {\it tropical line} (see \cite{SpSt} for a detailed description of the Grassmannian of tropical lines). Note that $L$ is an $n$-tree where the leaves of $L$ (labelled by $1,\ldots,n$) are the points at infinity of the unbounded edges (where the leaf $i$ is lying on the ray with direction $e_i$) and it satisfies the following property: if an edge of $L$ gives rise to a partition $I$ and $I^c=\{1,\ldots,n\}\setminus I$ of the leaf set, then the direction of this edge (towards the leaves in $I$) is equal to $e_I=\sum_{i\in I}\,e_i$. Each point $c=(c_1,\ldots,c_n)\in L$ gives rise to a tropical plane curve $\mathcal{T}(F_c)$ with support $\mathcal{A}$, so we can consider $L$ as the parameter space of a {\it linear pencil} and by abuse of notation, we will also denote this linear pencil by $L$. Note that $L$ can be seen as the image under the valuation map of a linear pencil of plane curves over $K$.

The {\it fixed locus} of the linear pencil $L$ is the set of points $P\in\TP^2$ such that each curve in $L$ goes through $P$. Unlike as for linear pencils over $K$, the fixed locus of $L$ is not determined by the intersection of two different curves of $L$. The following result says that each point $P$ in the fixed locus of $L$ corresponds to some curve of $L$, for which $P$ is a special point.

\begin{theorem} \label{thm P fixed}
Let $L\subset\TP^{n-1}$ be a linear pencil of tropical plane curves with support set $\mathcal{A}=\{a_1,\ldots,a_n\}$. Then $P\in\TP^2$ is a point of the fixed locus of $L$ if and only if there exists a point $c\in L$ such that one of the following two cases holds:
\begin{enumerate}
\item[(1)] there exist elements $i,j,k,\ell\in\{1,\ldots,n\}$ such that the pairs of leaves $\{i,j\}$ and $\{k,\ell\}$ belong to different components of $L\setminus\{c\}$ and the minimum of $\{c_i+a_i\cdot P\}_{i=1,\ldots,n}$ is attained by the terms corresponding to $i,j,k,\ell$.
\item[(2)] there exist elements $i,j,k\in\{1,\ldots,n\}$ such that the leaves $i,j,k$ belong to different components of $L\setminus\{c\}$ (thus $c$ is a vertex of $L$) and the minimum of $\{c_i+a_i\cdot P\}_{i=1,\ldots,n}$ is attained by the terms corresponding to $i,j,k$.
\end{enumerate}
\end{theorem}

Let $C=\{P_1,\ldots,P_{n-2}\}$ be a configuration of $n-2$ points in $\TP^2$, with $P_k=(x_k,y_k,z_k)$. We say that $C$ is {\it general} with respect to $\mathcal{A}$ if $C$ does not lie on a tropical projective curve with support $\mathcal{A}\setminus \{a_i,a_j\}$ for any pair $i,j$. This condition can easily be checked as follows. Let $M\in\mathbb{R}^{(n-2)\times n}$ be the matrix with entries $M_{k,\ell}=a_{\ell}\cdot P_k=r_{\ell}x_k+s_{\ell}y_k+t_{\ell}z_k$. For any $i,j\in\{1,\ldots,n\}$, let $M^{(i,j)}$ be the maximal minor of $M$ that we get by erasing the $i$-th and $j$-th column of $M$. The {\it tropical determinant} of $M^{(i,j)}$ is defined as $$\tropdet(M^{(i,j)})=\bigoplus_{\sigma\in\mathcal{S}_{ij}}\left(\bigotimes_{k=1}^{n-2} M_{k,\sigma(k)}\right) = \min_{\sigma\in\mathcal{S}_{ij}}\left( \sum_{k=1}^{n-2} a_{\sigma(k)}\cdot P_k\right),$$ where $\mathcal{S}_{ij}$ is the set of bijections $\sigma:\{1,\ldots,n-2\}\to\{1,\ldots,n\}\setminus\{i,j\}$. Then $C$ is general with respect to $\mathcal{A}$ if and only if each maximal minor $M^{(i,j)}$ is {\it tropically non-singular}, i.e. the minimum in $\tropdet(M^{(i,j)})$ is attained only once (see \cite[Theorem 5.3]{RGST}).

If $C$ is general, then the set of tropical plane curves with support $\mathcal{A}$ passing through $C$ is a tropical line $L_C\subset \TP^{n-1}$ and hence a linear pencil. Note that $L_C$ is the intersection of the tropical hypersurfaces $\mathcal{T}(\mathcal{H}_j)\subset\TP^{n-1}$ (i.e. the points $(x_1,\ldots,x_n)\in\TP^{n-1}$ where the minimum in $\mathcal{H}_j$ is attained at least twice) with $\mathcal{H}_j=\min_{i=1,\ldots,n}\{a_i\cdot P_j+x_i\}$ and $j\in\{1,\ldots,n-2\}$. In case $C$ is not general, this intersection does not need to be a tropical line. Instead we can consider the stable intersection of the tropical hypersurfaces $\mathcal{T}(\mathcal{H}_j)$, where loosely speaking only the transverse intersections of the faces are taken into account. This is a tropical line $L_C$ and is called the {\it stable linear pencil} of tropical plane curves through $C$.

In \cite{RGST}, it is proved that the for each configuration $C$ of $n-2$ points the linear pencil $L_C$ is {\it compatible} with $\mathcal{A}$, i.e. if $(ij|kl)$ is a trivalent subtree of $L_C$, then the convex hull of $a_i,a_j,a_k,a_l$ has at least one of the segments $\conv(a_i,a_j)$ or $\conv(a_k,a_l)$ as an edge. We will give a proof of an inverse implication, solving an open question raised in \cite{RGST}.

\begin{theorem} \label{thm main}
Let $L$ be a tropical projective line in $\TP^{n-1}$ that is compatible with $\mathcal{A}$. Assume that $L$ is trivalent and that each trivalent vertex $v$ of $L$ corresponds to a maximal subdivision $\Delta(v)$ of $\conv(\mathcal{A})$. Then there exists a general configuration $C$ of $n-2$ points in $\TP^2$ such that $L=L_C$. In particular, each point in $C$ corresponds to a trivalent vertex of $L$ (in the sense of Theorem \ref{thm P fixed}).
\end{theorem}

As a corollary, we have the following result.

\begin{theorem} \label{thm corollary main}
For each combinatorial type $\mathcal{T}$ of trivalent $n$-trees that are compatible with $\mathcal{A}$, we can find a general configuration $C$ such that $L_C$ is of type $\mathcal{T}$.
\end{theorem}

For example, if all the points $a_1,\ldots,a_n$ are on the boundary of $\conv(\mathcal{A})$, then there are precisely $\frac{1}{n-1}{2n-4\choose n-2}$ combinatorial types of trivalent trees (out of the $(2n-5)!!=1\cdot3\cdot5\cdot\ldots\cdot(2n-5)$ in total) that are compatible with $\mathcal{A}$ (see \cite[Corollary 6.4]{RGST}) and each type can be obtained by a linear pencil $L_C$ with $C$ general.

A short outline of the rest of the paper is as follows. In Section \ref{section fixed}, we will focuss on the fixed locus of tropical linear pencils. In order to show Theorem \ref{thm P fixed}, we will give a characterization of tropical lines lying in a tropical hyperplane or one of its skeletons (see Lemma \ref{lemma skeleton}). In Section \ref{section compatible}, we will study linear pencils that are compatible with $\mathcal{A}$ and give the proofs of Theorem \ref{thm main} and \ref{thm corollary main}.

\section{Fixed locus of a linear pencil} \label{section fixed}

\begin{example} \label{ex square}
Let $\mathcal{A}=\{(0,0,2),(1,0,1),(0,1,1),(1,1,0)\}$. In Figure \ref{figure examples}, four examples of linear pencils of tropical plane curves with support set $\mathcal{A}$ are pictured. The fixed locus of $L_1$, $L_2$, $L'_2$ and $L_{\infty}$ is respectively $\{(0,0,0)\}$, $\{(0,-1,0),(1,1,0)\}$, $\{(0,0,0),(1,1,0)\}$ and the line segment $[(0,0,0),(0,1,0)]$.
\begin{figure}[ht] \label{figure examples}
\centering
\includegraphics[height=3.8cm]{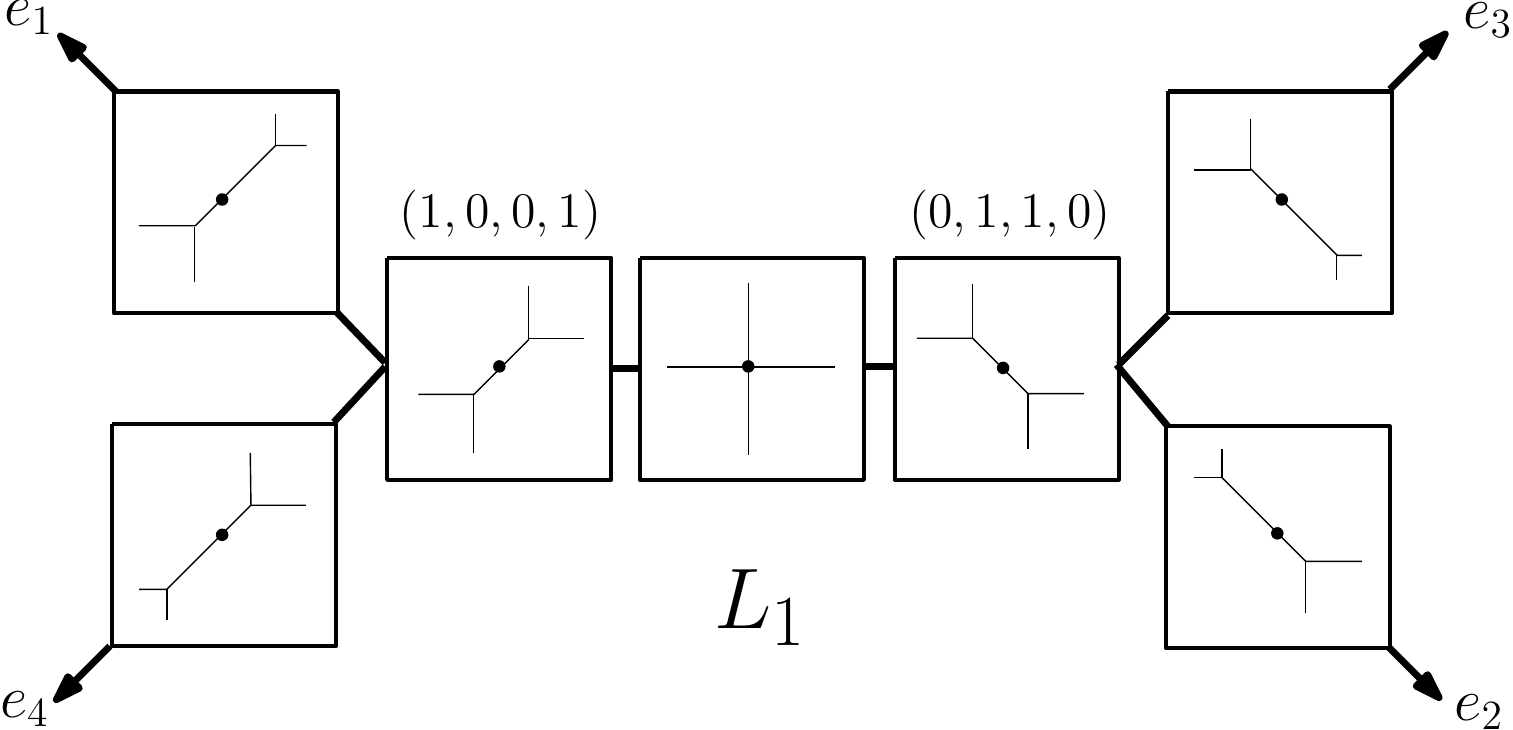} \hspace{5mm}
\includegraphics[height=3.8cm]{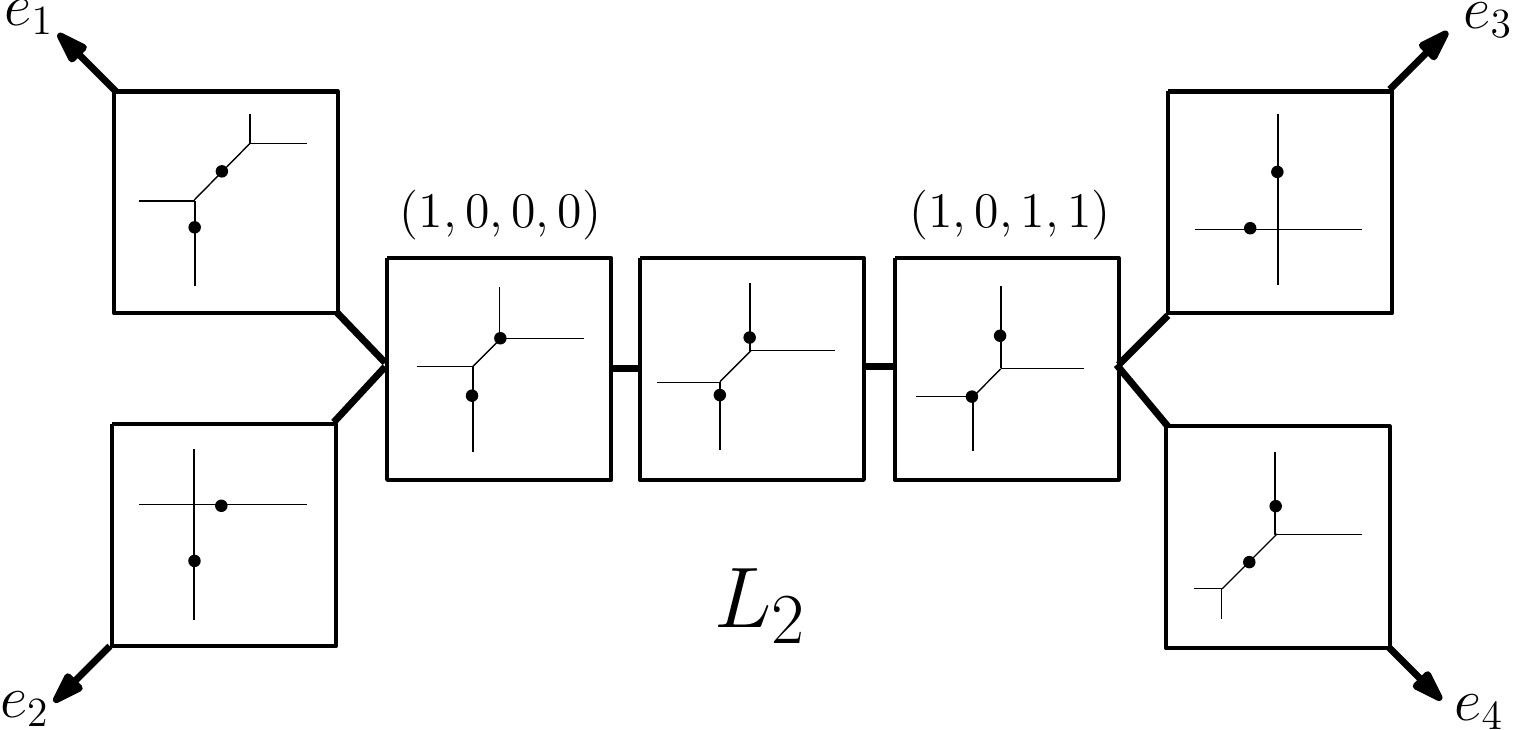} \\ \vspace{5mm}
\includegraphics[height=3.8cm]{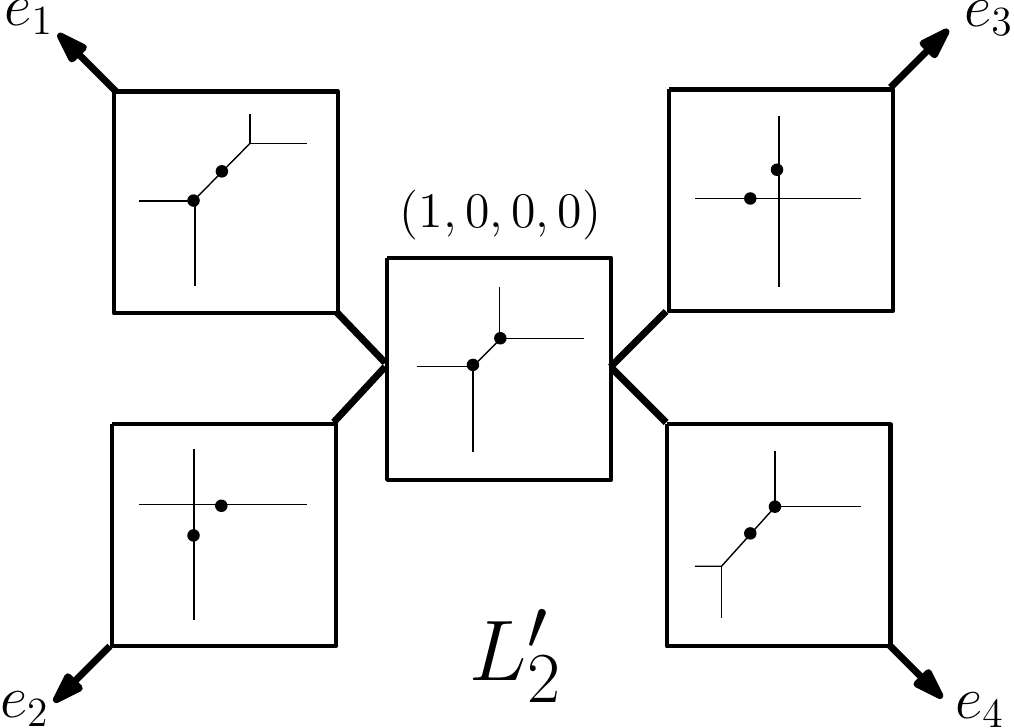} \hspace{5mm}
\includegraphics[height=3.8cm]{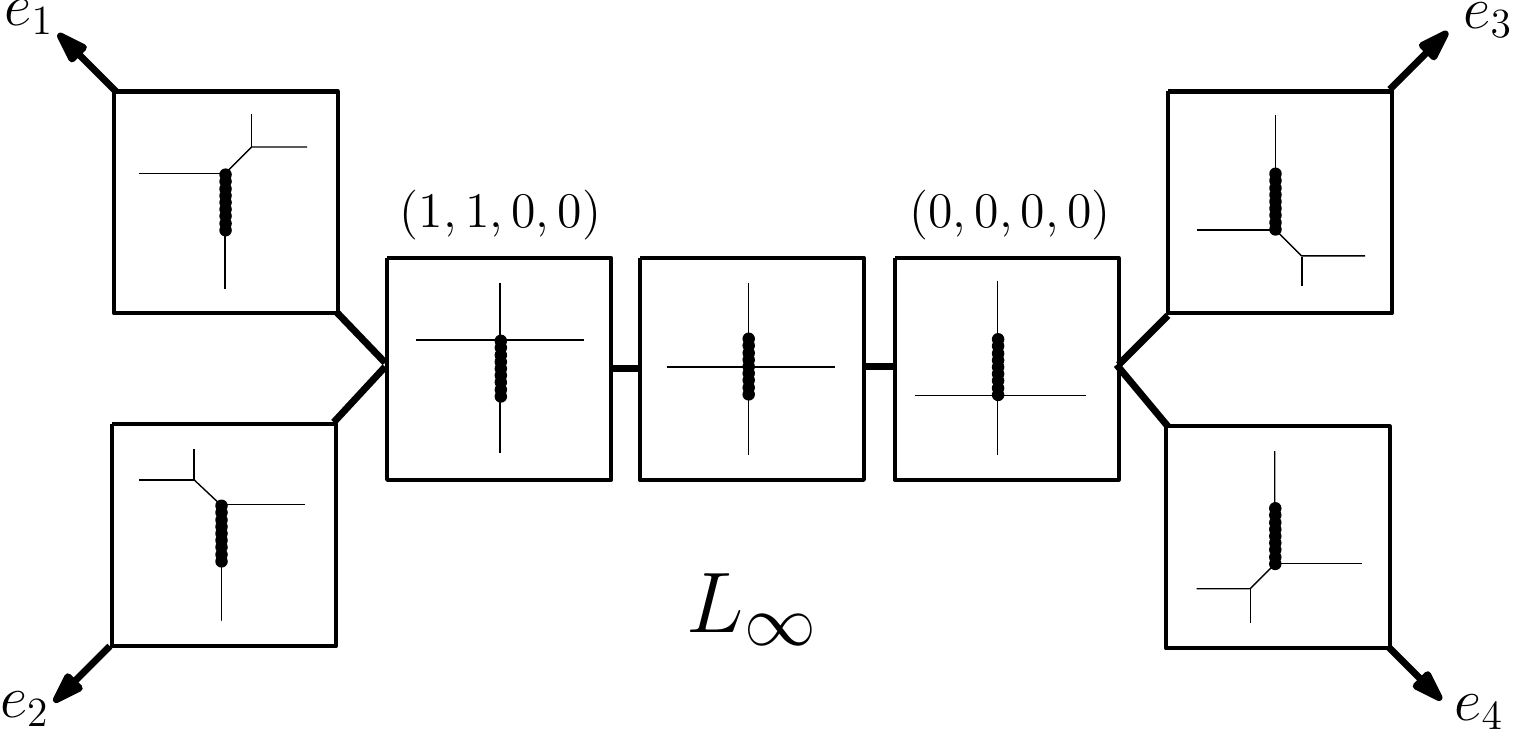}
\caption{the linear pencils $L_1, L_2, L'_2$ and $L_{\infty}$}
\end{figure}

One can see (using Theorem \ref{thm P fixed}) that the fixed locus of a linear pencil corresponding to $\mathcal{A}$ has always $1$, $2$ or infinitely many points.
\end{example}

Let $P$ be a point in the fixed locus of a linear pencil $L$ of tropical plane curves with support set $\mathcal{A}$. Then for each $c\in L\subset\TP^{n-1}$, the minimum of the set $$\{c_i+a_i\cdot P\}_{i=1,\ldots,n}=\{c_i+r_ix+s_iy+t_iz\}_{i=1,\ldots,n}$$ is attained at least twice, hence the set $L+\mathcal{A}\cdot P:=\{c+\mathcal{A}\cdot P\,|\,c\in L\}$ (where $\mathcal{A}\cdot P=(a_i\cdot P)_{i=1,\ldots,n}\in \TP^{n-1}$) is contained in the standard hyperplane $\mathcal{T}(x_1\oplus\ldots\oplus x_n)\subset\TP^{n-1}$. Note that $L+\mathcal{A}\cdot P$ is a translation of $L$, so it is also a tropical line in $\TP^{n-1}$.

If $t\in \{1,\ldots,n\}$, let $\Pi_t\subset \TP^{n-1}$ be the subset consisting of points $(x_1,\ldots,x_n)$ such that $\min_{i=1}^n \{x_i\}$ is attained at least $t$ times. Note that $\Pi_t$ is the $(n-t)$-dimensional skeleton of the standard tropical hyperplane $\Pi_2=\mathcal{T}(x_1\oplus \ldots \oplus x_n)\subset \TP^{n-1}$ (if $t\geq 2$) and that $\Pi_1=\TP^{n-1}$. In the following lemma, we examine tropical lines $L$ lying in some $\Pi_t$.

\begin{lemma} \label{lemma skeleton}
Let $\Gamma$ be a tropical line in $\TP^{n-1}$. If $I\subset\{1,\ldots,n\}$, let $$\Pi(I)=\{(x_1,\ldots,x_n)\in \TP^{n-1}\,|\,x_i=\min_{j=1}^n x_j \text{ for all }i\in I\}$$ and $\Pi(\Gamma,I)=\Gamma\cap \Pi(I)$.
\begin{enumerate}
\item[(a)]
There exists a point $\pi_{\Gamma}\in \Gamma$ such that $\Pi(\Gamma,I)$ is the union of $\{\pi_{\Gamma}\}$ with the components of $\Gamma\setminus\{\pi\}$ having no leaf in $I$, for each subset $I\subset\{1,\ldots,n\}$ with $\Pi(\Gamma,I)\neq\emptyset$. Moreover, if $\Gamma\subset \Pi_t$ and if $\pi$ is an $m$-valent point of $\Gamma$, then $\pi\in \Pi_{\left\lceil \frac{m}{m-1}t\right\rceil}$.
\item[(b)]
Let $p$ be an $m$-valent point of $\Gamma$ and let $C_1,\ldots,C_m$ be the components of $\Gamma\setminus\{p\}$. Assume that $p\in\Pi(\Gamma,I)$ for some subset $I\subset\{1,\ldots,n\}$ and that $|I\setminus I_j|\geq t$ for all $j\in\{1,\ldots,m\}$, where $t\geq 1$ is an integer and $I_j=I\cap C_j$. Then we have that $p=\pi_{\Gamma}$ and $\Gamma\subset \Pi_t$.
\end{enumerate}
\end{lemma}

\begin{proof}
This proof is subdivided into seven steps. 
%
%
Denote by $\Gamma_I$ the minimal subtree of $\Gamma$ containing the leaves in $I$. \\

\noindent \textsf{Step 1}: $\Gamma_I\cap \Pi(\Gamma,I)$ is at most a singleton. \\
\textsf{Proof of step 1}:
Assume that $x$ and $y$ are two different points in the intersection $\Gamma_I\cap \Pi(\Gamma,I)$. Note that this implies that $I$ is not a singleton. We can find $i$ and $j$ in $I$ such that the path between the leaves $i$ and $j$ passes through $x$ and $y$ (and assume $x$ is closest to the leaf $i$). We can take tropical coordinates such that $x_i=y_i$ and $x_j<y_j$. This is in contradiction with the equalities $x_i=x_j$ and $y_i=y_j$ since $x,y\in \Pi(\Gamma,I)$. \hfill $\diamond$ \\

\noindent \textsf{Step 2}: Let $x\in \Pi(\Gamma,I)$ and denote by $\gamma(I)\in \Gamma_I$ the point that is closest to $x$. If $y\in \Gamma$ and the path between $y$ and $\gamma(I)$ passes through $x$, then also $y\in \Pi(\Gamma,I)$. \\
\textsf{Proof of step 2}:
We can take tropical coordinates of $x$ and $y$ such that $$\left\{\begin{array}{l l} y_i = x_i & \text{for all } i\in I\\ y_i=x_i+\epsilon_i \text{ with }\epsilon_i\geq 0& \text{for all } i\not\in I\end{array}\right.$$ Then $y_i=x_i=x_j=y_j$ for all $i,j\in I$ and $y_i=x_i\leq x_j\leq x_j+\epsilon_j=y_j$ for all $i\in I$ and $j\not\in I$, hence $y\in \Pi(\Gamma,I)$. \hfill $\diamond$ \\

\noindent \textsf{Step 3}: The set $\Pi(\Gamma,I)$ is connected and $\gamma(I)\in \Gamma_I$ is independent of $x\in \Pi(\Gamma,I)$. \\
\textsf{Proof of step 3}:
First we prove that $\gamma(I)\in \Gamma_I$ is independent of the choice of the point in $\Pi(\Gamma,I)$. So let $x_1,x_2\in \Pi(\Gamma,I)$ and denote by $y_1$ and $y_2$ the points in $\Gamma_I$ closest to respectively $x_1$ and $x_2$. Assume that $y_1\neq y_2$. There exist elements $i$ and $j$ in $I$ such that $y_1$ and $y_2$ are contained in the path between the leaves $i$ and $j$ of $\Gamma$ (and assume $y_1$ is closest to the leaf $i$). We can take tropical coordinates of $x_1,x_2,y_1,y_2\in\TP^{n-1}$ such that $$(x_1)_i=(x_2)_i=(y_1)_i=(y_2)_i$$ and $$(x_1)_j=(y_1)_j<(x_2)_j=(y_2)_j.$$ This is in contradiction with $x_1,x_2\in\Pi(\Gamma,I)$.

Now let $x_1,x_2\in \Pi(\Gamma,I)$. We are going to prove that the path between $x_1$ and $x_2$ is contained in $\Pi(\Gamma,I)$ as well. For this, it is enough to show that the point $x$ on the path between $x_1$ and $x_2$ that is closest to $\gamma(I)$ is contained in $\Pi(\Gamma,I)$ (using Step 2). Let $K_i$ (for $i\in\{1,2\}$) be the leaves in the component of $\Gamma\setminus\{x\}$ containing $x_i$ and let $K=\{1,\ldots,n\}\setminus(K_1\cup K_2)$. Note that $I\subset K$. We can take tropical coordinates of $x,x_1,x_2$ such that $$\left\{\begin{array}{l l} (x_1)_k = x_k & \text{for all } k\in K\cup K_2\\ (x_1)_k=x_k+\epsilon_k \text{ with }\epsilon_k\geq 0& \text{for all } k\in K_1\end{array}\right.$$ and $$\left\{\begin{array}{l l} (x_2)_k = x_k & \text{for all } k\in K\cup K_1\\ (x_2)_k=x_k+\epsilon'_k \text{ with }\epsilon'_k\geq 0& \text{for all } k\in K_2\end{array}\right..$$ Let $i\in I$. Then we have that $$\left\{\begin{array}{l l} x_k=(x_1)_k=(x_2)_k=(x_1)_i=(x_2)_i=x_i & \text{if } k\in I\\ x_k=(x_1)_k=(x_2)_k\geq (x_1)_i=(x_2)_i=x_i & \text{if } k\in K\setminus I \\ x_k=(x_2)_k\geq (x_2)_i=x_i & \text{if } k\in K_1 \\ x_k=(x_1)_k\geq (x_1)_i=x_i & \text{if } k\in K_2 \end{array}\right.$$ and hence $x\in \bigcap_{i\in I}\,\Pi(\Gamma,\{i\})=\Pi(\Gamma,I)$.  \hfill $\diamond$ \\

If $I\subset\{1,\ldots,n\}$ and $\Pi(\Gamma,I)$ is non-empty, we denote by $\pi(I)$ the point in $\Pi(\Gamma,I)$ closest to $\gamma(I)$. The existence of such a point follows from Step 3. \\

\noindent \textsf{Step 4}: If $I'\subset I\subset \{1,\ldots,n\}$ with $I'\neq \emptyset$ and $\Pi(\Gamma,I)\neq\emptyset$, then we have that $\pi(I')=\pi(I)$. \\
\textsf{Proof of step 4}:
First, we give a description of the set $\Pi(\Gamma,\{i\})$ for $i\in I$, where we need to consider the cases $\pi(I)=\gamma(I)$ and $\pi(I)\neq \gamma(I)$ separately (see Figure \ref{figure two cases}). Note that $\Pi(\Gamma,\{i\})\neq \emptyset$ since $\Pi(\Gamma,I)\subset\Pi(\Gamma,\{i\})$.

\begin{figure}[ht] \label{figure two cases}
\centering
\includegraphics[height=3.8cm]{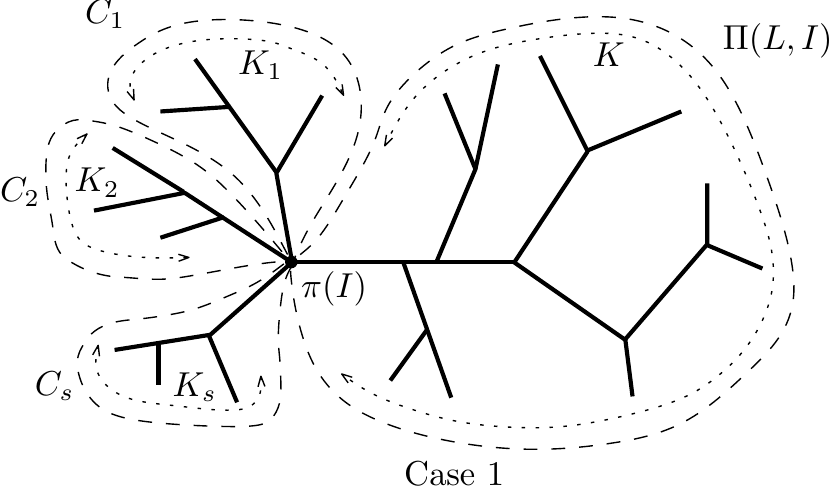} \hspace{1cm} \includegraphics[height=3.8cm]{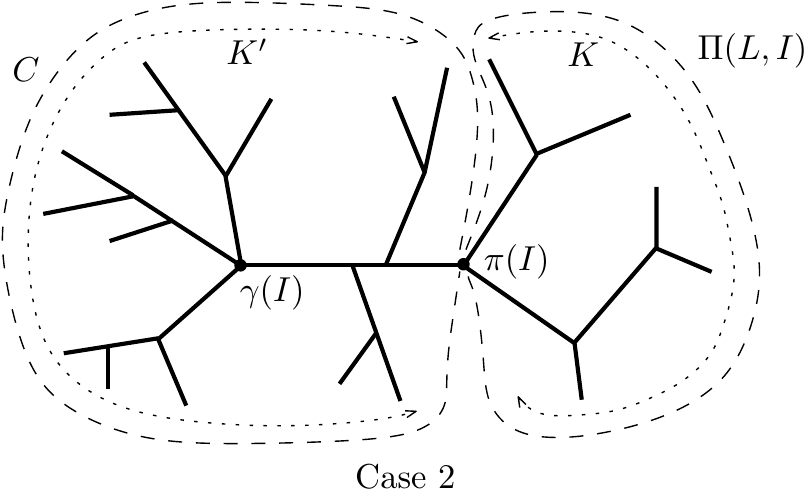}
\caption{the two cases in Step 4 of Lemma \ref{lemma skeleton}(f)}
\end{figure}

In the first case (so $\pi(I)=\gamma(I)$), let $C_1,\ldots,C_s$ be the different components of $\Gamma\setminus \Pi(\Gamma,I)$. Let $K_j$ be the set of leaves in $C_j$ and $K$ the set of leaves in $\Pi(\Gamma,I)$. Note that $I\subset K_1\cup \ldots \cup K_s$, $I\cap K_j\neq \emptyset$ (by Step 2) and $s>1$ (since $\Gamma_I$ is a subtree). Let $i\in K_j\cap I$. From Step 2, it follows that $$\Pi(\Gamma,I)\cup\bigcup_{r\neq j}C_r\subset \Pi(\Gamma,\{i\})$$ since $\pi(I)\in\Pi(\Gamma,\{i\})$. Assume that this inclusion is strict, hence there exists a point $x\in C_j$ such that $x\in \Pi(\Gamma,\{i\})$. Using Steps 2 and 3, we may assume that $x$ is on the path between the leaf $i$ and $\pi(I)$. Take $i'\in I$ such that the leaf $i'$ is not contained in $K_j$, hence $x$ and $\pi(I)$ are contained in $\Gamma_{\{i,i'\}}\cap \Pi(\Gamma,\{i,i'\})$, which is in contradiction with Step 1. So we have that $\Pi(\Gamma,\{i\})=\Pi(\Gamma,I)\cup \bigcup_{r\neq j} C_r$.

In the second case (so $\pi(I)\neq\gamma(I)$), let $C=\Gamma\setminus \Pi(\Gamma,I)$ and denote the set of leaves in $C$ and $\Pi(\Gamma,I)$ by respectively $K'$ and $K$. Note that $I\subset K'$. If $i\in I$ and $\Pi(\Gamma,\{i\})\neq \Pi(\Gamma,I)$, consider a point $x\in \Pi(\Gamma,\{i\})\setminus\Pi(\Gamma,I)$ sufficiently close to $\pi(I)$ (to be precise, let $x$ be an inner point of a finite edge of $\Gamma$ which also contains $\pi(\Gamma)$). We can take tropical coordinates of $x$ and $\pi(I)$ such that $$\left\{\begin{array}{l l} \pi(I)_k=x_k & \text{for all } k\in K' \\ \pi(I)_k=x_k+\epsilon & \text{for all } k\in K \end{array}\right..$$ If $k\in I$, then $x_k=\pi(I)_k=\pi(I)_i=x_i$, hence $x\in\Pi(\Gamma,I)$, a contradiction. Hence, $\Pi(\Gamma,\{i\})=\Pi(\Gamma,I)$.

In both cases, we have that $\pi(\{i\})=\pi(I)$, which implies the equality $\pi(I')=\pi(I)$ for each non-empty subset $I'\subset I$. \hfill $\diamond$ \\

\noindent \textsf{Step 5}: the point $\pi(I)$ is independent of the set $I\subset\{1,\ldots,n\}$ with $\Pi(\Gamma,I)\neq 0$. \\
\textsf{Proof of step 5}:
We may assume that $\Gamma\subset \Pi_t$ but $\Gamma\not\subset \Pi_{t+1}$ for some $t$.
Take $J\subset\{1,\ldots,n\}$ with $|J|=t$ such that $|\Pi(\Gamma,J)|>1$ and such that $\Pi(\Gamma,J)\subset\Pi(\Gamma,\{k\})$ implies $k\in J$. Consider an edge of $\Gamma$ containing $\pi(J)$ that is not fully contained in $\Pi(\Gamma,J)$ and consider points on this edge close to $\pi(J)$ but not in $\Pi(\Gamma,J)$. These points must be contained in some set $\Pi(\Gamma,K)$ with $|K|=t$ and $K\neq J$, so they must be contained in a set $\Pi(\Gamma,\{k\})$ with $k\not\in J$. Since the set $\Pi(\Gamma,\{k\})$ is closed, we get that $\pi(J)\in\Pi(\Gamma,\{k\})$. If the leaf $k$ would not be contained in $\Pi(\Gamma,J)$, this would imply that $\Pi(\Gamma,\{k\})\supset\Pi(\Gamma,J)$ using Step 2, hence $k\in J$, a contradiction. So $k\in\Pi(\Gamma,J)$ and again Step 2 implies $\Gamma=\Pi(\Gamma,J)\cup\Pi(\Gamma,\{k\})$. On the other hand, since $$\{\pi(J)\}\subset\Pi(\Gamma,J)\cap\Pi(\Gamma,\{k\})=\Pi(\Gamma,J\cap\{k\}),$$ we can use Step 4 and see that $\pi(J\cup\{k\})=\pi(J)=\pi(\{k\})$.
Now let $I\subset\{1,\ldots,n\}$ be a subset such that $\Pi(\Gamma,I)\neq\emptyset$. Since $\Pi(\Gamma,I\cup J)=\Pi(\Gamma,I)\cap\Pi(\Gamma,J)$ or $\Pi(\Gamma,I\cup\{k\})=\Pi(\Gamma,I)\cap\Pi(\Gamma,\{k\})$ is non-empty, we have that either $\pi(I)=\pi(J)$ or $\pi(I)=\pi(\{k\})$ by Step 4, hence $\pi(I)=\pi(J)=\pi(\{k\})$.  \hfill $\diamond$ \\

\noindent \textsf{Step 6}: part (a) \\
\textsf{Proof of step 6}:
The existence of a point $\pi_{\Gamma}$ follows from Step 5. Now assume that $\pi_{\Gamma}$ is $m$-valent and that $\Gamma\subset \Pi_t$. Denote the components of $\Gamma\setminus\{\pi_{\Gamma}\}$ by $C_1,\ldots,C_m$. If $r\in\{1,\ldots,m\}$, let $I_r$ be the set of leaves $i$ in $C_r$ with $\Pi(\Gamma,\{i\})\neq\emptyset$, and let $\alpha_r=|I_r|$. Note that $\pi_{\Gamma}\in\Pi(\Gamma,\bigcup_{r=1}^m I_r)$, so it suffices to show that $\alpha_1+\ldots+\alpha_m\geq \frac{m}{m-1}t$. Since $C_r\not\subset \Pi(\Gamma,\{i\})$ for each $i\in I_r$, we have that $(\alpha_1+\ldots+\alpha_m)-\alpha_r\geq t$. Taking the sum of these equations for all $r\in\{1,\ldots,m\}$, we get that $(m-1)(\alpha_1+\ldots+\alpha_m)\geq mt$ and the statement follows. \hfill $\diamond$ \\

\noindent \textsf{Step 7}: part (b) \\
\textsf{Proof of step 7}:
Assume that $p\neq \pi_{\Gamma}$ en let $C_j$ be the component containing $\pi_{\Gamma}$. Since $|I\setminus I_j|\geq t\geq 1$, it follows that $I\setminus I_j$ is non-empty, so we can find an element $i\in I_k$ with $k\neq j$. Then (a) implies that $p\not\in\Pi(\Gamma,\{i\})$, thus $p\not\in \Pi(\Gamma,I)$, a contradiction. In order to show that $\Gamma\subset\Pi_t$, note that $C_j\cup\{p\}\subset\Pi(\Gamma,I\setminus I_j)$ with $|I\setminus I_j|\geq t$. \hfill $\diamond$
\end{proof}

We can apply the above lemma in the context of tropical linear pencils.

\begin{proof}[Proof of Theorem \ref{thm P fixed}]
We already showed that $P$ is contained in the fixed locus of $L$ if and only if $\Gamma=L+\mathcal{A}\cdot P$ is contained in the standard tropical hyperplane $\Pi_2\subset\TP^{n-1}$. By Lemma \ref{lemma skeleton}, this is equivalent with the existence of a point $\pi_{\Gamma}\in \Gamma$ such that one of the following two cases holds: either $\pi_{\Gamma}\in\Pi(\Gamma,\{i,j,k,\ell\})$ where the pairs of leaves $\{i,j\}$ and $\{k,\ell\}$ are contained in different components of $\Gamma\setminus\{\pi_{\Gamma}\}$, or $\pi_{\Gamma}\in\Pi(\Gamma,\{i,j,k\})$ where the leaves $i,j,k$ are contained in different components of $\Gamma\setminus\{\pi_{\Gamma}\}$. Note that $$\left\lceil \frac{2m}{m-1}\right\rceil=\left\{\begin{array}{l l} 4 & \text{if } m=2\\ 3 & \text{if } m\geq 3 \end{array}\right..$$ We conclude that $c=\pi_{\Gamma}-\mathcal{A}\cdot P\in L$ satisfies the conditions of the theorem.
\end{proof}

\begin{proposition}
Let $L\subset\TP^{n-1}$ be the stable linear pencil of tropical plane curves with support set $\mathcal{A}$ through the points $P_1,\ldots,P_{n-2}\in\TP^2$. Then for each vertex $v$ of $L$ there exists a point $P_j$ such that the minimum of $\{v_i+a_i\cdot P_j\}_{i=1,\ldots,n}$ is attained at least $3$ times.
\end{proposition}
\begin{proof}
We know that $L$ is the stable intersection of the tropical hyperplanes $$\mathcal{T}((a_1\cdot P_i)\otimes x_1 \oplus \ldots \oplus (a_n\cdot P_i)\otimes x_n)\subset\TP^{n-1}.$$ Let $v$ be a vertex of $L$. Assume that the minimum of $\{v_i+a_i\cdot P_j\}_{i=1,\ldots,n}$ is attained precisely two times for each $P_j$ and let $\alpha(j),\beta(j)\in\{1,\ldots,n\}$ be the indices of the terms corresponding to the minimum. Then $L$ is locally defined by the stable intersection of the hyperplanes given by $a_{\alpha(j)}\cdot P_j+x_{\alpha(j)}=a_{\beta(j)}\cdot P_j+x_{\beta(j)}$ in a neighborhood of $v$. Since such an intersection is a linear subspace, we get a contradiction.
\end{proof}

So if $L$ is the stable linear pencil of tropical plane curves through some configuration of points $P_i$, the above result gives some evidence that each $P_i$ is situated in Case (2) of Theorem \ref{thm P fixed}. In fact, we will prove this in the next section.

\section{Compatible linear pencils} \label{section compatible}

Let $L\subset\TP^{n-1}$ be a tropical line that is compatible with $\mathcal{A}$. A vertex split of $L$ at a trivalent vertex $v\in L$ gives rise to partition $\mathcal{A}_1,\mathcal{A}_2,\mathcal{A}_3$ of $\mathcal{A}$. The compatibility condition implies that for each quartet $a,b,c,d\in\mathcal{A}$ with $a,b\in\mathcal{A}_i$ and $c,d\not\in\mathcal{A}_i$ (with $i\in\{1,2,3\}$) holds that the convex hull of the four points $a,b,c,d$ has at least one of the segments $\conv(a,b)$ or $\conv(c,d)$ as an edge. Denote this property on quartets by $(\maltese)$.

\begin{lemma} \label{lemma comb}
Let $\mathcal{A}_1,\mathcal{A}_2,\mathcal{A}_3\subset \mathcal{A}$ be a partition that satisfies $(\maltese)$.
Then for each maximal triangulation $\Delta=\{T_1,\ldots,T_r\}$ of $\conv(\mathcal{A})$ (with corners of the triangles in $\mathcal{A}$), there exists a triangle $T\in\Delta$ having one vertex in each of the sets $\mathcal{A}_1,\mathcal{A}_2,\mathcal{A}_3$.
\end{lemma}

\begin{proof}
First we claim that for each pair $i,j\in\{1,2,3\}$ we have that $\conv(\mathcal{A}_i)\subset\conv(\mathcal{A}_j)$, $\conv(\mathcal{A}_j)\subset\conv(\mathcal{A}_i)$ or $\conv(\mathcal{A}_i)\cap\conv(\mathcal{A}_j)=\emptyset$. Indeed, assume that $\conv(\mathcal{A}_i)\setminus\conv(\mathcal{A}_j)$, $\conv(\mathcal{A}_j)\setminus\conv(\mathcal{A}_i)$ and $\conv(\mathcal{A}_i)\cap\conv(\mathcal{A}_j)$ are nonempty, then there exist elements $a\in \mathcal{A}_i\cap\conv(\mathcal{A}_j)$ and $b\in\mathcal{A}_i\setminus \conv(\mathcal{A}_j)$. The line segment $\conv(a,b)$ cuts the border of $\conv(\mathcal{A}_j)$ at a line segment $\conv(c,d)$ with $c,d\in\mathcal{A}_j$. The convex hull of the quartet $a,b,c,d\in\mathcal{A}$ has $\conv(a,b)$ or $\conv(c,d)$ as diagonals, a contradiction.
Using the above claim, there are three possible configurations of $\mathcal{A}_1,\mathcal{A}_2,\mathcal{A}_3$ (after a renumbering if necessary): $\conv(\mathcal{A}_1),\conv(\mathcal{A}_2),\conv(\mathcal{A}_3)$ pairwise disjoint, $(\conv(\mathcal{A}_1)\cup\conv(\mathcal{A}_2))\subset \conv(\mathcal{A}_3)$ or
$\conv(\mathcal{A}_1)\subset\conv(\mathcal{A}_2)\subset\conv(\mathcal{A}_3)$. The latter configuration is impossible. Indeed, take $c\in\mathcal{A}_1$ and $d\in\mathcal{A}_3\setminus \conv(\mathcal{A}_2)$. Then the line segment $\conv(c,d)$ passes through a edge $\conv(a,b)$ of the border of $\conv(\mathcal{A}_2)$ and the quartet $a,b,c,d$ is in contradiction with $(\maltese)$.

Assume we are in the case of the first configuration, so the the sets $\conv(\mathcal{A}_i)$ are pairwise disjoint. Since $\Delta$ is a triangulation of $\conv(\mathcal{A})$ (and thus not only of the sets $\conv(\mathcal{A}_i)$), we may assume (after a renumbering if necessary) there is an edge of a triangle $T\in\Delta$ connecting a point in $\mathcal{A}_1$ with point in $\mathcal{A}_2$. Let $\Delta'\subset\Delta$ be the set of all triangles having two vertices in  $\mathcal{A}_1$ and one vertex in $\mathcal{A}_2$ or vice versa. If $\Delta'=\emptyset$, then the third vertex of the triangle $T$ is contained in $\mathcal{A}_3$ and the statement follows. Now assume $\Delta'$ is nonempty, say $\Delta'=\{T_1,\ldots,T_s\}$ (after a renumbering if necessary). We can draw a line $\Lambda$ subdividing the plane into two half-planes $H_1$ and $H_2$ such that $\conv(\mathcal{A}_1)\subset H_1$, $\conv(\mathcal{A}_2)\subset H_2$ and $\Lambda\cap\mathcal{A}=\emptyset$. Using an identification of the line $\Lambda$ with $\mathbb{R}$ (using an affine coordinate), each of the triangles $T_i\in \Delta'$ cuts $\Lambda$ in a closed interval $I_i$ of $\mathbb{R}$ and two intervals $I_i$ and $I_j$ (with $i,j\in\{1,\ldots,s\}$ different) are either pairwise disjoint or subsequent (i.e. the intersection is a point). We may assume that $T_1,\ldots,T_s$ are numbered in such a way that $\max(I_i)\leq \min(I_{i+1})$ for all $i\in\{1,\ldots,s-1\}$. If one of the above inequalities is strict, then the statement follows, since the edge of $T_i$ containing $\max(I_i)$ has to be the edge of another triangle $T\in\Delta$ having its third vertex in $\mathcal{A}_3$. Hence the intervals $I_1,\ldots,I_s$ are subsequent. Denote by $a_{1,i}\in\mathcal{A}_1$ and $a_{2,i}\in\mathcal{A}_2$ the two vertices of the edge of $T_i$ containing $\min(I_i)\in \Lambda$ and by $a_{1,s+1}\in\mathcal{A}_1$ and $a_{2,s+1}\in\mathcal{A}_2$ the two vertices of the edge of $T_s$ containing $\max(I_s)$. Note that either $a_{1,i}=a_{1,i+1}$ and $T_i=\conv(a_{1,i},a_{2,i},a_{2,i+1})$ or $a_{2,i}=a_{2,i+1}$ and $T_i=\conv(a_{1,i},a_{1,i+1},a_{2,i})$. If $\conv(a_{1,1},a_{2,1})$ (resp. $\conv(a_{1,s+1},a_{2,s+1})$) is not a part of the border of $\conv(\mathcal{A})$, then this line segment is the edge of a triangle $T\in\Delta$ different from $T_1$ (resp. $T_s$) with one vertex in each of the three sets $\mathcal{A}_i$. So we may assume that $\conv(a_{1,1},a_{2,1})$ and $\conv(a_{1,s+1},a_{2,s+1})$ are parts of the border of $\conv(\mathcal{A})$.
\begin{figure}[ht]
\centering
\includegraphics[height=3.5cm]{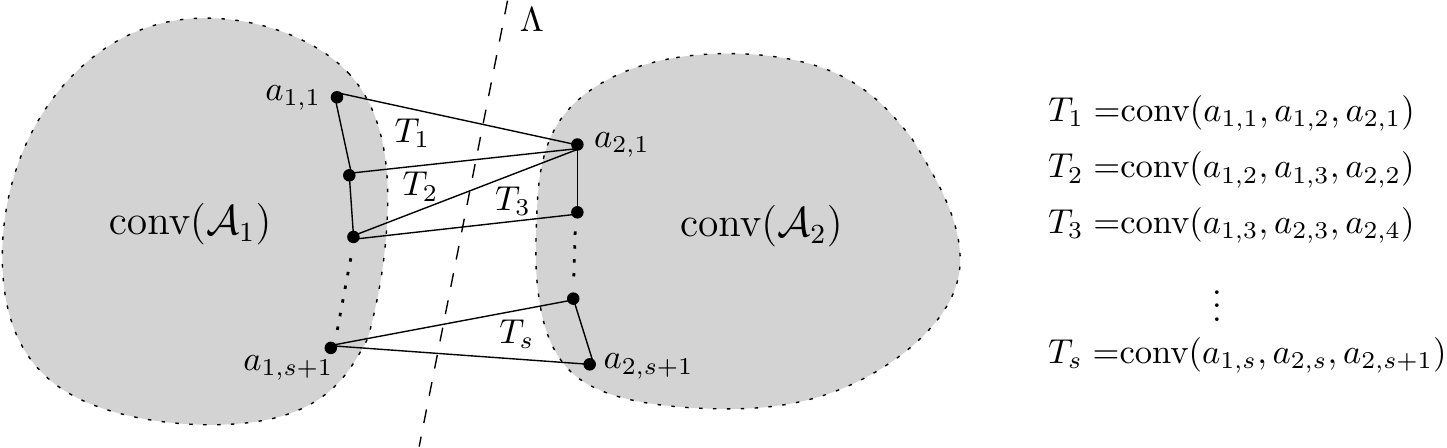}
\caption{first configuration}
\end{figure}

First note that $\conv(a_{1,1},a_{1,s+1},a_{2,1},a_{2,s+1})$ does not contain a point of $\mathcal{A}_3$, since
\begin{IEEEeqnarray}{l} \conv(a_{1,1},a_{1,s+1},a_{2,1},a_{2,s+1}) \quad \subset \quad \conv(a_{1,1},\ldots,a_{1,s+1},a_{2,i},\ldots,a_{2,s+1}) \nonumber \\
\qquad \subset \quad \conv(a_{1,1},\ldots,a_{1,s+1}) \cup \conv(a_{2,1},\ldots,a_{2,s+1}) \cup \bigcup_{i=1}^s T_i \nonumber \\
\qquad \subset \quad \conv(\mathcal{A}_1)\cup\conv(\mathcal{A}_2)\cup \bigcup_{i=1}^s T_i. \nonumber
\end{IEEEeqnarray}
Let $a_3\in \mathcal{A}_3$. We may assume that $a_3$ is contained in the half-plane $H_1$. If $a_{1,1}=a_{1,s+1}$, then $a_3$ is contained in $\conv(a_{1,1},a_{2,1},a_{2,s+1})$ since $H_1\cap\conv(\mathcal{A})=H_1\cap\conv(a_{1,1},a_{2,1},a_{2,s+1})$, a contradiction. So we have that $a_{1,1}\neq a_{1,s+1}$ and the quartet $a_{1,1},a_{1,s+1},a_{2,1},a_3$ is in contradiction with $(\maltese)$.

To finish the proof, we have to take care of the second configuration, so $(\conv(\mathcal{A}_1)\cup\conv(\mathcal{A}_2))\subset \conv(\mathcal{A}_3)$. Assume there is no triangle $T\in\Delta$ having one vertex in $\mathcal{A}_1$ and one in $\mathcal{A}_2$. For $(i,j)\in\{(3,0),(2,1),(1,2)\}$, denote by $\Delta_{i,j}\subset\Delta$ the set of all triangles $T$ having $i$ vertices in $\mathcal{A}_1$ and $j$ vertices in $\mathcal{A}_3$. For a triangle $T=\conv(a_1,a_3,a'_3)\in\Delta_{1,2}$ with $a_1\in\mathcal{A}_1$ and $a_3,a'_3\in\mathcal{A}_3$, let $C(T)\subset\mathbb{R}^2$ be the cone with top $a_1$ over the line segment $[a_3,a'_3]$. Then we can see that $$\left(\mathbb{R}^2\setminus \bigcup_{T\in\Delta_{1,2}} C(T)\right) \subset  \left(\bigcup_{T\in\Delta_{3,0}} T \cup \bigcup_{T\in\Delta_{2,1}} T \right) \subset \left(\conv{\mathcal{A}_1} \cup \bigcup_{T\in\Delta_{2,1}} T \right).$$
\begin{figure}[ht]
\centering
\includegraphics[height=3cm]{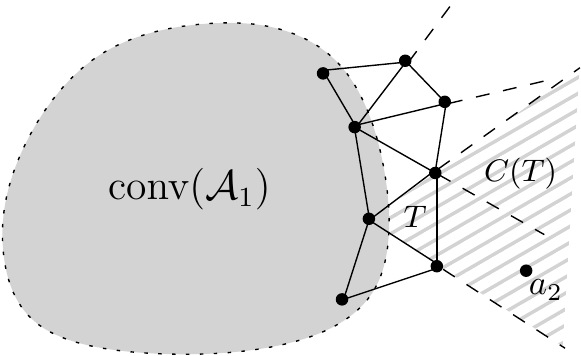}
\caption{second configuration}
\end{figure}

If $a_2\in\mathcal{A}_2$, then $a_2\not\in \conv(\mathcal{A}_1) \cup \bigcup_{T\in\Delta_{2,1}} T$ and thus $a_2\in C(T)$ for some triangle $T=\conv(a_1,a_3,a'_3)\in\Delta_{1,2}$. Hence the quartet of points $a_1,a_2,a_3,a'_3$ is in contradiction with $(\maltese)$. So there exists a triangle $T\in\Delta$ connecting $\mathcal{A}_1$ with $\mathcal{A}_2$ and by using analogous arguments as in the case of the first configuration, we can prove the existence of a triangle $T\in\Delta$ having one vertex in each of the sets $\mathcal{A}_1,\mathcal{A}_2,\mathcal{A}_3$.
\end{proof}

\begin{proposition} \label{prop P_v}
Let $L\subset \TP^{n-1}$ be a linear pencil of tropical plane curves with support $\mathcal{A}$ such that $L$ is compatible with $\mathcal{A}$. Let $v$ be a trivalent vertex of $L$ such that the regular subdivision $\Delta(v)$ of $\conv(\mathcal{A})$ is maximal. Then $v$ corresponds to a point $P_v$ in the fixed locus of $L$.
\end{proposition}

\begin{proof}
By Lemma \ref{lemma comb}, we find a triangle $T_v=\conv(a_i,a_j,a_k)$ of the maximal subdivision $\Delta(v)$ of $\mathcal{A}$ such that $a_i\in \mathcal{A}_1$, $a_j\in\mathcal{A}_2$ and $a_k\in\mathcal{A}_3$. Let $P_v=(x,y,z)\in\mathcal{T}(F_v)\subset\TP^2$ be the point which is dual to the triangle $T_v$, hence $$v_i+r_ix+s_iy+t_iz=v_j+r_jx+s_jy+t_jz=v_k+r_kx+s_ky+t_kz<v_{\ell}+r_{\ell}x+s_{\ell}y+t_{\ell}z$$ for all $\ell\in\{1,\ldots,n\}\setminus\{i,j,k\}$. From Theorem \ref{thm P fixed}, it follows that $P_v$ is contained in the fixed locus of $L$.
\end{proof}

Let $L\subset \TP^{n-1}$ be a linear pencil of tropical plane curves with support $\mathcal{A}$. Assume that $L$ is compatible with $\mathcal{A}$, trivalent and that each (trivalent) vertex of $L$ gives rise to a maximal subdivision of $\conv(\mathcal{A})$. Note that Proposition \ref{prop P_v} implies that the fixed locus of $L$ contains a configuration of $n-2$ points and in order to prove Theorem \ref{thm main}, we need to show that this configuration is general.

We are going to assign a subgraph $G_v$ of $K_{n-2,n}$ to each vertex $v$ of $L$. We identify the vertices of $K_{n-2,n}$ by the vertices of $L$ (of which there are $n-2$) and the elements of $\mathcal{A}$ (of which there are $n$). The edge $(w,a_{\ell})$ of $K_{n-2,n}$ is and edge of $G_v$ if and only if $a_{\ell}$ is a corner of $T_w$ and the path between the leaf $\ell$ and $v$ passes through $w$. So $(w,a_{\ell})\in E(G_v)$ if and only if $v_{\ell}+a_{\ell}\cdot P_w=\min_{k=1,\ldots,n}\,\{v_k+a_k\cdot P_w\}$. Note that precisely two of the three corners of $T_w$ satisfy this condition for $w\neq v$, and all three for $w=v$. If $a_{\ell}\in\mathcal{A}$, the vertex $w$ of $L$ that is closest to the leaf $\ell$ gives rise to an edge $(w,a_{\ell})\in E(G_v)$. We conclude that $|E(G_v)|=2(n-3)+3=2n-3$ and $|V(G_v)|=n+(n-2)=2n-2$.

\begin{lemma}
The graph $G_v$ is connected and has no cycles.
\end{lemma}

\begin{proof}
We claim that $G_v$ is connected. First note that it is sufficient to show that each vertex $w\in V(G_v)$ is connected with $v$. Let $\Lambda_w$ be the union of $\{w\}$ with the components of $L\setminus\{w\}$ not containing $v$. We will prove by induction on the number $\lambda_w$ of vertices in the subtree $\Lambda_w$ that each vertex $w'$ in $\lambda_w$ is connected with $w$ in $G_v$. Since $\Lambda_v=L$, this would imply the claim. If $\lambda_w=1$, there is nothing to prove. Now assume the statement is proven for any subtree $\Lambda_w$ with $\lambda_w<\lambda$ and take a subtree $\Lambda_w$ with $\lambda_w=\lambda$. Let $w'\neq w$ be a vertex in $\Lambda$. Consider the vertex $w''\neq w$ closest to $w$ on the path between $w$ and $w'$ (so $w'\in\Lambda_{w''}$) and take the corner $a_\ell$ of $T_w$ corresponding to a leaf $\ell$ lying in $\Lambda_{w''}$. Denote by $w'''$ the vertex of $\Lambda_{w''}\subset L$ which is closest to $\ell$. Then $(w''',a_{\ell}),(w,a_{\ell})\in E(G_v)$, hence $w$ and $w'''$ are connected in $G_v$. On the other hand, by the induction hypothesis on the subtree $\Lambda_{w''}$, both $w'$ and $w'''$ are connected with $w''$. We conclude that $w$ is connected with $w'$ in $G_v$.
\begin{figure}[ht]
\centering
\includegraphics[height=3cm]{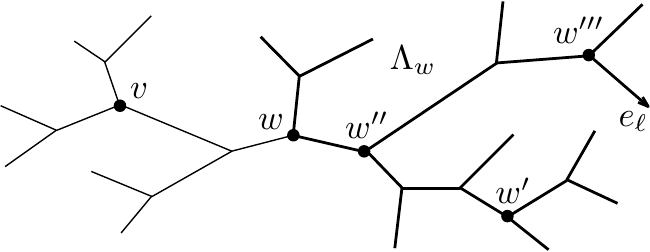}
\caption{the subtree $\Lambda_w$ of $L$}
\end{figure}

Since $G_v$ is connected, we can compute the genus (or first Betti number) of $G_v$ as $g(G_v)=|E(G_v)|-|V(G_v)|+1=0$. This implies that $G_v$ has no cycles.
\end{proof}

In the same way as we introduced $G_v\subset K_{n-2,n}$ for vertices $v$ of $L$, we can define $G_c\subset K_{n-2,n}$ for $2$-valent points $c$ of $L$, i.e. $(w,a_{\ell})\in E(G_c)$ if and only if $c_{\ell}+a_{\ell}\cdot P_w=\min_{k=1,\ldots,n}\,\{c_k+a_k\cdot P_w\}$. Note that if $c$ is contained in an edge adjacent to $v$, then $V(G_c)=V(G_v)$ and $E(G_c)=E(G_v)\setminus \{(v,a_{\ell})\}$ where $a_{\ell}$ is the corner of $T_v$ that corresponds to the leaf $\ell$ of $L$ lying on the same side of $v$ as $c$. This implies that $G_c$ has genus $0$ (and two connected components). If $\mathcal{B}\subset\mathcal{A}$, denote by $N(\mathcal{B})$ the {\it neighborhood} of $\mathcal{B}$ in $G_c$, i.e. the set of vertices $v$ of $L$ such that $(v,a_{\ell})\in E(G_c)$ for some $a_{\ell}\in\mathcal{B}$.

\begin{lemma} \label{lemma neighborhood}
Let $i,j\in\{1,\ldots,n\}$ be leaves of $L$ and $c$ be a $2$-valent point of $L$ on the path between $i$ and $j$. Then the inequality $|N(\mathcal{B})|\geq |\mathcal{B}|$ holds for each subset $\mathcal{B}\subset \mathcal{A}\setminus\{a_i,a_j\}$.
\end{lemma}

\begin{proof}
Let $C_i$ and $C_j$ be the components of $L\setminus\{c\}$ such that $i\in C_i$ and $j\in C_j$. We are going to prove this lemma by induction on $|\mathcal{B}|$. If $\mathcal{B}=\emptyset$, there is nothing to prove. Now assume that $\mathcal{B}\neq \emptyset$ and let $\mathcal{B}_i=\mathcal{B}\cap C_i$ and $\mathcal{B}_j=\mathcal{B}\cap C_j$. Since $\mathcal{B}=\mathcal{B}_i\cup\mathcal{B}_j$, we may assume that $\mathcal{B}_i\neq \emptyset$. Denote by $w$ the end point of the edge that contains $c$ and that is contained in $C_i$. For each vertex or leaf $v$ of $C_i$ different from $w$, write $\rho(v)\neq v$ to denote the vertex of $L$ on the path between $v$ and $c$ closest to $v$. Let $\mathcal{V}$ be the set of vertices or leaves $v$ of $C_i$ such that all leaves for which the path to $c$ passes through $v$ are contained in $\mathcal{B}_i$ and such that this is not the case for $\rho(v)$. We can take an element $v\in\mathcal{V}$ such that the component of $L\setminus\{\rho(v)\}$ not containing $c$ and $u$ does not have any leaf in $\mathcal{B}_i$. Note that the triangle $T_{\rho(v)}$ has precisely one corner $a_{\ell}\in\mathcal{B}_i$ and that $(\rho(v),a_{\ell})\in E(G_c)$, hence $\rho(v)\in N(\mathcal{B})$. On the other hand, we see that $\rho(v)\not\in N(\mathcal{B}\setminus\{a_{\ell}\})$, so the induction hypothesis implies that $$|N(\mathcal{B})|\geq |N(\mathcal{B}\setminus\{a_{\ell}\})|+1\geq |\mathcal{B}\setminus\{a_{\ell}\}|+1=|\mathcal{B}|.$$
\end{proof}

\begin{proof}[Proof of Theorem \ref{thm main}]
Since $L$ is a trivalent tree with $n$ leaves, it has precisely $n-2$ trivalent vertices which we will denote by $v_1,\ldots,v_{n-2}$. Using Proposition \ref{prop P_v}, each $v_i$ gives rise to a point $P_i$ in the fixed locus of $L$. It suffices to show that the configuration $C=\{P_1,\ldots,P_{n-2}\}$ is general with respect to $\mathcal{A}$. Therefore, we need to show that each maximal minor $M^{(i,j)}$ of the matrix $M$ with $M_{k,\ell}=a_{\ell}\cdot P_k$ is tropically non-singular.

Let $c$ be a $2$-valent point of $L$ on the path between $i$ and $j$ and consider the graph $G_c\subset K_{n-2,n}$. Hall's Marriage Theorem (see \cite{Hall}) and Lemma \ref{lemma neighborhood} imply that there is a matching in $G_c$ between $\{v_1,\ldots,v_{n-2}\}$ and $\mathcal{A}\setminus\{a_i,a_j\}$. Note that this matching is unique since $G_c$ contains no cycles. Let $\psi$ be the bijection in $S_{ij}$ that corresponds to the matching. We claim that the minimum in $\tropdet(M^{(i,j})$ is attained only by the term that corresponds to $\psi$. Indeed, let $\sigma \in S_{ij}$ be different from $\psi$. For all $k\in\{1,\ldots,n-2\}$, we have that $c_{\sigma(k)}+a_{\sigma(k)}\cdot P_k\geq c_{\psi(k)}+a_{\psi(k)}\cdot P_k$ and equality holds if and only if $(v_k,a_{\sigma(k)})\in E(G_c)$. If we take the sum of all these inequalities (and erase the term $\sum_{\ell\in\{1,\ldots,n\}\setminus\{i,j\}} c_{\ell}$), we see that $$\sum_{k=1}^{n-2} a_{\sigma(k)}\cdot P_k > \sum_{k=1}^{n-2} a_{\psi(k)}\cdot P_k ,$$ since at least one of the inequalities must be strict. This proves the claim and the theorem.
\end{proof}

\begin{remark}
The linear pencils $L'_2$ and $L_{\infty}$ in Example \ref{ex square} show that the two extra conditions on $L$ in Theorem \ref{thm main}, i.e. $L$ is trivalent and each vertex gives rise to a maximal subdivision of $\conv(\mathcal{A})$, are necessary. Note that $L'_2$ and $L_{\infty}$ are the stable linear pencils corresponding to respectively $\{(0,0,0),(1,1,0)\}$ and $\{0,0,0),(0,1,0)\}$. In fact, both linear pencils can be seen as a limit case of linear pencils  satisfying the two conditions and thus corresponding to general configurations, but the limit of these configurations is non-general.
\end{remark}

\begin{proof}[Proof of Theorem \ref{thm corollary main}]
Let $\Delta$ be a maximal regular subdivision of $\conv(\mathcal{A})$ and consider the facet of the secondary fan (see \cite{DLRS}) corresponding to $\Delta$, i.e. the set of all points $(c_1,\ldots,c_n)\in \mathbb{R}^n$ such that the projection to the last coordinate of the lower faces of the polytope $$\conv((a_1,c_1),\ldots,(a_n,c_n))\subset \mathbb{R}^4$$ is $\Delta$. We can take a tree $L\subset \TP^{n-1}$ of type $\mathcal{T}$ such that all its finite edges are contained in this facet (by taking the edge lengths small enough). Now the statement follows from Theorem \ref{thm main}.
\end{proof}


\begin{thebibliography}{9999}
\bibitem[BrSt]{BrSt} S. Brodsky, B. Sturmfels, {\it Tropical Quadrics Through Three Points}, to appear in Linear Algebra Appl.
\bibitem[DLRS]{DLRS} J.A. De Loera, J. Rambau, F. Santos, {\it Triangulations: Structures for Algorithms and Applications}, Springer series ``Algorithms and Computation in Mathematics'', Vol. 25, 2010, 539p.
\bibitem[GaMa]{GaMa} A. Gathmann, H. Markwig, {\it The Caporaso-Harris formula and plane relative Gromov-Witten invariants in tropical geometry}, Math. Ann. 338 (2007), 845-868.
\bibitem[Hall]{Hall} P. Hall, {\it On Representatives of Subsets}, J. London Math. Soc. 10 (1935), 26-30.
\bibitem[Mikh]{Mikh} G. Mikhalkin, {\it Enumerative tropical algebraic geometry in $\mathbb{R}^2$}, J. Amer. Math. Soc. 18 (2005), no. 2, 313-377.
\bibitem[RGST]{RGST} J. Richter-Gebert, B. Sturmfels, T. Theobald, \textit{First steps in tropical geometry}, Idempotent mathematics and mathematical physics, 289-317, Contemp. Math., 377, Amer. Math. Soc., Providence, RI, 2005.
\bibitem[SpSt]{SpSt} D. Speyer, B. Sturmfels, {\it The tropical Grassmannian}, Adv. Geom. 4 (2004), 389-411.
\end{thebibliography}
\end{document}